\documentclass[final,leqno,onefignum,onetabnum]{siamltex1213}
\usepackage{algorithm,amsmath,graphicx,epstopdf,graphics,epsf}
\usepackage{caption,float,multicol,bm,amssymb}
\usepackage[export]{adjustbox}
\begin{document}
\newtheorem{thm}{Theorem}
\newtheorem{defn}{Definition}
\newtheorem{cor}{Corollary}
\newtheorem{eg}{Example}
\newtheorem{lem}{Lemma}
\def\ni{\noindent}
\def\vs{\vspace}
\setlength{\belowcaptionskip}{-15pt}

\title{On residual norms in the Rayleigh-Ritz and refined projection methods \thanks{This work is supported by the National Board of Higher Mathematics, India under Grant number 02/40(3)/2016. }}

\author{Mashetti Ravibabu \thanks{Department of Computational and Data Sciences, Indian Institute of Science,
Bangalore,  India, 560012. (\email{mashettiravibabu2@gmail.com}).} }

\maketitle

\slugger{sisc}{xxxx}{xx}{x}{x--x}%slugger should be set to mms, siap, sicomp, sicon, sidma, sima, simax, sinum, siopt, sisc, or sirev

\begin{abstract}
This paper derives bounds for the ratio of residual norms in the refined and  Rayleigh-Ritz projection methods.
% for solving linear eigenvalue problems. 
To do this, it uses the Least squares and line search projection method proposed in \cite{Mallannagift}. 
%Unlike the bounds appeared in the literature, 
The bound derived in this paper is less costly to compute. Further, it is practically useful to assess the superiority of the refined and the Rayleigh-Ritz projection methods one over the other. 
\end{abstract}



\pagestyle{myheadings} \thispagestyle{plain} \markboth{On residual norms in the Rayleigh-Ritz and the refined projection methods}{Mashetti Ravibabu}

%$$\tau-1 =(\tau-\alpha_7)g(\tau+\alpha_6/2)-\tau/\tau-\alpha_7$$

\section{Introduction}
Projection methods are quite familiar to solve large sparse eigenvalue problems.  These methods produce eigenpair approximations using  either oblique or orthogonal projections onto a specifically chosen vector space.
Depending on a chosen vector space these methods are classified as Krylov subspace methods and Jacobi-Davidson type methods. The Lanczos method for symmetric matrices and the Arnoldi method for non-symmetric matrices are well-known and come under the category of Krylov subspace methods. Similar to Lanczos and Arnoldi methods, Jacobi-Davidson method also starts with an arbitrarily chosen unit vector called the Initial Vector.  Then at each iteration,  it extends an existing vector space using the solution of a  system of linear equations called the Correction Equation. The correction equation varies depending on the procedure chosen for extracting eigenpair approximations from a vector space. 
% 
% Projection methods are quite familiar for solving large sparse eigenvalue problems. As the name suggests, these methods use either oblique or orthogonal projections onto a vector space chosen specifically, to produce eigenpair approximations. Basing on a vector space chosen these methods are classified as Krylov subspace methods and Jacobi-Davidson type methods.
% 
% The well-known  Krylov subspace methods are the Lanczos method for symmetric matrices and the Arnoldi method for non-symmetric matrices. Similar to Lanczos and Arnoldi methods, Jacobi-Davidson method also starts with an arbitrarily chosen unit vector, called the ``Initial vector''.
%Then at each iteration, it uses  the  solution of a linear system of equations called the ``correction equation" to extend an existing vector space. The 
%%coefficient matrix in the 
%correction equation varies depending on the procedure chosen for extracting eigenpair approximations from a vector space.

The Rayleigh-Ritz projection is a well-known procedure for extracting eigenpair approximations and is inherent in these projection methods \cite{saad,ste}.  
The Rayleigh-Ritz projection produces good approximations to the eigenvalues in the exterior of the spectrum. To better approximate interior eigenvalues, It requires the inverse of a given matrix which is computationally more costly.  This problem resolved by using the Harmonic projection \cite{HP1}, but, as in the Rayleigh-Ritz projection, the eigenvector approximations produced by the Harmonic projection method also may not converge to an eigenvector, even though the corresponding eigenvalue approximations do converge \cite{jiac,HPrec}. This misconvergence problem is avoidable in the Refined projection method \cite{someshwaragift, feng}, and in the  Least squares and Line search technique (LLS). 

%Rayleigh-Ritz projection is a well known inherent procedure
%% to get an eigenpair approximation from the vector space chosen already             
%for extracting eigenpair approximations in the projection methods
%%from a vector space
%\cite{saad,ste}. It provides good approximations to those eigenvalues in the exterior of the spectrum. To produce better
%approximations to interior eigenvalues, it has to work with the inverse of a given matrix. But, this is
%computationally more costly. 

%This problem is resolved by using the 
%Harmonic projection  \cite{HP1}. 
%%It is a familiar method to approximate eigenvalues in the interior of the spectrum. 
%However, Like Rayleigh-Ritz projection, it also have a drawback that eigenvector
%approximations may not converge to an eigenvector,
%even though the corresponding eigenvalue approximations do converge
%\cite{jiac,HPrec}. The way to evite this misconvergence problem is to use either a
%Refined projection \cite{someshwaragift, feng} or Least squares and Line search technique (LLS). 

The refined projection method preserves an eigenvalue approximation that obtained using Rayleigh-Ritz projection. Then, it determines corresponding eigenvector approximation such that residual norm is minimum overall unit vectors in a vector space, from which an eigenvalue approximation sought. To find such an approximate eigenvector, it solves a singular value problem of smaller size. The LLS technique procures an approximate eigenpair from Rayleigh-Ritz projection. Then, it improves an eigenvector approximation in Rayleigh-Ritz projection by using least squares heuristics and line search technique \cite{Someshwaragift2,Mallannagift}.

%The refined projection method preserves an eigenvalue approximation
%that obtained using Rayleigh-Ritz projection. Then, it determines corresponding eigenvector approximation such that 
%%with an optimal property that 
%residual norm is minimum over all
%unit vectors in a vector space, from which an eigenvalue approximation sought. To find such an approximate eigenvector, it solves a singular value problem of smaller size.
%
%Similar to the refined projection, the LLS technique also procures an eigenvalue approximation from the Rayleigh-Ritz projection. In addition, it also procures an eigenvector approximation and improves it 
%by using least squares heuristics and line search technique.
%It has been demonstrated in \cite{Someshwaragift2,Mallannagift} that for the Arnoldi and the Jacobi-Davidson methods, LLS technique is more efficient than the refined projection.

It is a general belief that the residual norm in the refined projection method is too small compared to that in the Rayleigh-Ritz projection. However, it is still unknown how much smaller the former residual norm compared to the later. Although the answer for this will be quite useful to create robust and efficient eigensolvers, it is still unanswered as these two methods came from different perspectives, singular value problem, and eigenvalue problem respectively. This paper extinguishes this question by deriving bounds for the ratio of residual norms in Rayleigh-Ritz and refined projections.
%It is a general belief that residual norm in the refined projection method  is too small
%compared to residual norm of
%an eigenvector approximation in the Rayleigh-Ritz projection. In contrast, it is still unknown how smaller the former residual norm compared to the later. Even though, the answer for this will be quite useful to create robust and efficient eigensolvers, it is still unanswered as these two methods came from the different perspectives,  singular value problem and eigenvalue problem respectively. This paper extinguishes this question by deriving bounds for the ratio of residual norm in the Rayleigh-Ritz projection to  residual norm in the refined projection.
%%For this, we exploit an eigenvector approximation in the LLS method as a bridge to interlink those eigenvector approximations in the Rayleigh-Ritz and refined projections.

The paper is organized as follows; Section 2 briefly discusses the Rayleigh-Ritz projection, Refined projection, and LLS methods.  Then, Section 3 determines the upper and lower bounds for the concerned ratio of residual norms. Section 4 concludes the paper.
%This paper is organized as follows; Section~2 briefly discusses the Rayleigh-Ritz projection,
%refined projection and LLS methods. In Section~3, we derive the bounds for the ratio of a residual norm in the Rayleigh-Ritz projection to a residual norm in the refined projection.
%%, which is a concern of this paper.
%Section~4 concludes the paper.
\section{Rayleigh-Ritz, Refined Rayleigh-Ritz and LLS methods}
Let $A$ be a given matrix of order $n$ and ${\cal V}$ is a $k$ dimensional vector space. 
Suppose that column vectors of a matrix  $V = [v_1~ v_2~ \cdots v_k]$ form an orthonormal basis of ${\cal V}.$  Then to produce an approximate eigenpairs of $A,$ the Rayleigh-Ritz projection method 
solves an eigenvalue problem for the matrix $H:=V^\ast AV$ of order $k.$ In general, $k$ is small and this eigenvalue problem can be solved using classical methods such as the QR algorithm. 

As column vectors of $V$ are orthonormal,
%the orthonormality of the s of $V$ implies 
$V^\ast V =I.$ Further, an
eigenpair $(\theta_i,y_i)$ of $H$ satisfies the relation:
$V^\ast (A-\theta_i I)Vy_i =0.$
That means, approximations to
eigenpairs of $A$ produced by the Rayleigh-Ritz projection method satisfy the \textit{Galerkin condition:}
$$AVy_i -\theta_iVy_i \perp {\cal V}~~\mbox{for}~~ i = 1,2,\cdots ,k.$$
Equivalently, this can be written as follows:
$$AVy_i-\theta_iVy_i \perp v,\quad \forall v \in {\cal V},~i= 1,2 \cdots,k.$$
In general, the above equation is not a good indication on $Vy_i$ being an eigenvector approximation. It shows that $Vy_i$ is orthogonal to its corresponding 
%$\theta$ is an approximate eigenvalue and 
residual vector $AVy_i-\theta_iVy_i.$ It further shows that eigenvector approximations $Vy_i$ may not converge to an eigenvector of $A$, even though corresponding eigenvalue approximations  converge to an eigenvalue \cite{jiac}. %This problem arises especially, when small perturbations of a matrix $H$ produces a spurious eigenvalue close to the required eigenvalue. In this case an eigenvector approximation is ill-conditioned.
%\subsection{Refined projection}

The Refined projection method is a remedy for the mis-convergence problem of eigenvector approximations in Rayleigh- Ritz projection. An eigenvector approximation $u_{R_i}$ in the Refined projection method satisfies the
following:
\begin{equation}\label{E1}\relax
\|(A-\theta_i I)u_{R_i}\|=\min\{\|(A-\theta_i I)u\| : u \in {\cal
V}~\mbox{and}~\|u\| = 1\},
\end{equation}
%The potential remedy for mis-convergence problems of eigenvector 
%approximation in the Rayleigh- Ritz projection is the Refined projection.  In the refined projection, an eigenvector approximation $u_{R}$ has the least residual norm overall unit vectors in the vector space from which eigenvalue approximations sought. It thus satisfies the following:
%\begin{equation}\label{E1}\relax
%\|(A-\theta_i I)u_{R_i}\|=\min\{\|(A-\theta_i I)u\| : u \in {\cal
%V}~\mbox{and}~\|u\| = 1\},
%\end{equation}
where $\theta_i$ is an eigenvalue approximation that retained from the Rayleigh-Ritz projection. Thus, an eigenvector approximation $u_{R_i }$ has the least residual norm overall unit vectors in the vector space from which eigenvalue approximations sought. Hence, the 
refined projection method computes an eigenvector
approximation $u_{R_i}$ by solving a singular value problem for
$(A-\theta_i I)V$. It is a general belief that $\|(A-\theta_i I)u_{R_i}\|
\ll \|(A-\theta_i I)Vy_i\|.$ In this paper, we estimate  
$\|(A-\theta_i I)u_{R_i}\|/\|(A-\theta_i I)Vy_i\|$ by using a residual vector in the LLS method.

%Like the refined projection method, LLS method also At first, by using a solution $z_i$ of the following least squares problem: %correction done by reverso up to this.
%\begin{equation}\label{E2}\relax
%\|(A-\theta_i I)Vm_i\|=\min\{\|(A-\theta_i I)Vy_i+(A-\theta_i
%I)Vz\| : z \perp y_i\},
%\end{equation}
%it improves an  eigenvector approximation $Vy_i$ in the Rayleigh-Ritz projection to $Vm_i.$ It has been proved for the Arnoldi method that utilizing $Vm_i/\|Vm_i\|$ as an eigenvector approximation is more efficient than the refined projection \cite{Someshwaragift2}. 
The LLS method preserves an eigenvalue approximation from the Rayleigh-Ritz projection. Then, it solves  the following least squares problem to find a vector $z_i$:
\begin{equation}\label{E2}\relax
\|(A-\theta_i I)Vm_i\|=\min\{\|(A-\theta_i I)Vy_i+(A-\theta_i
I)Vz\| : z \perp y_i\},
\end{equation}
Further the LLS method uses the line search technique and updates an eigenvector approximation $Vy_i$ in Rayleigh-Ritz projection to $Vs_i/\|Vs_i\|.$ Therefore, an eigenvector approximation $Vs_i/\|Vs_i\|$ in the LLS method has the following optimal property:
\begin{equation}\label{E3}\relax
\frac{\|(A-\theta_i I)Vs_i\|^2}{\|s_i\|^2}~\mbox{
 is minimum over span}\{y_i,~(I-y_iy_i^\ast )z_i\}.
\end{equation}
%where $s_i = y_i+\tau_i
%(I-y_iy_i^\ast )z_i$ and $z_i$ is a solution vector of the least squares problem (\ref{E2}).

Note that eigenvector approximations in the  LLS method are explicitly related to those in the Rayleigh-Ritz projection, unlike eigenvector approximations in the refined projection method. Further,  The LLS method avoids the mis-convergence problem of eigenvector approximations in the Rayleigh-Ritz projection. The LLS method proved its better efficiency in the Jacobi-Davidson method compared to the Jacobi-Davidson method that inherently uses refined projection \cite{Mallannagift}.
%This improved method of eigenvector extraction proved more efficient than the refined projection for the Jacobi-Davidson method \cite{Mallannagift}. Further, it has been observed that like refined projection method the LLS method also
%avoids the mis-convergence problem of eigenvector approximations in the
%Rayleigh-Ritz projection. Note that, LLS method have an advantage that an eigenvector approximation in it is 
%%Even more, the eigenvector approximations in LLS method are 
%explicitly related to an  approximate
%eigenvector in the Rayleigh-Ritz projection, unlike  eigenvector approximation in the refined projection method.
The following section will first establish a few relations between residual norms in the refined projection and the LLS method. 
%Those relations will be
%useful to relate the residual norms in the 
%Rayleigh-Ritz and refined projection methods. Next, it derives bounds for the
%ratio of residual norm in Rayleigh-Ritz projection to the residual norm in the refined projection method, which is a concern of this paper.

In what follows, the Ritz value fixed  as $\theta$ and its corresponding
eigenvector approximations as  $Vy, V u_R, \frac{Vm}{\|Vm\|}$ and
$\frac{Vs}{\|Vs\|}$ in the Rayleigh-Ritz, refined and the LLS methods respectively. Further, a subscript notation $i$ in this section will be ignored.

\section{Comparison of residual norms}
The following theorem derives a relation between residual norms in the Rayleigh-Ritz and the least squares part of the LLS methods via using a
matrix $(A-\theta I)V.$
 %Include theorem
\begin{thm}\label{thm1}\relax
Let the minimization problem (\ref{E2}) have a non-zero solution
vector $(I-yy^\ast)z.$ Then the following is true:
\begin{equation}\label{E4}\relax
\|(A-\theta I)Vm\|^2 = \|(A-\theta I)Vy\|^2-\|(A-\theta
I)V(I-yy^\ast)z\|^2.
\end{equation}
\end{thm}
\begin{proof}
From the equation~(\ref{E2}) note that a vector $(I-yy^\ast )z$ 
minimizes the least squares functional $\|(A-\theta I)Vy+(A-\theta
I)V(I-yy^\ast)z\|^2.$ Thus, it is a solution of the following normal
equations:
\begin{eqnarray}\label{E5}\relax
\scalebox{0.92}{$(I-yy^\ast )V^\ast (A-\theta I)^\ast (A-\theta I)V(I-yy^\ast )z = -(I-yy^\ast )V^\ast (A-\theta I)^\ast  (A-\theta I)Vy.$}
\end{eqnarray}
Taking an inner product with  $z$ on both sides of the equation (\ref{E5}) gives
\begin{equation}\label{E6}\relax
\langle(A-\theta I)V(I-yy^\ast)z, (A-\theta I)Vy\rangle  =
-\|(A-\theta I)V(I-yy^\ast )z\|^2.
\end{equation}
From the equation (\ref{E2}) note that $Vm =
Vy+V(I-yy^\ast)z.$ Thus,
\begin{multline*}
\|(A-\theta I)Vm\|^2=\|(A-\theta I)Vy\|^2+\|(A-\theta
I)V(I-yy^\ast)z\|^2\\+ 2~ Real \langle(A-\theta I)V(I-yy^\ast)z, (A-\theta I)Vy\rangle .
\end{multline*}
Therefore, by using the equation (\ref{E6}), the above equation proves the equation~(\ref{E4}).
\end{proof}

The Theorem-\ref{thm1} shows that the least squares approach in the LLS method reduces residual norm to a better extent than the Rayleigh-Ritz method, provided $\|(A-\theta I)V(I-yy^\ast)z\|^2$ is large. 
Next, the following lemma \cite[Lemma-3]{Mallannagift} will be helpful in the Theorem-\ref{thm1bnew} to see that the line search technique of the LLS method will bring a further reduction in the residual norms.
\begin{lem}\label{thm1b}\relax
Let $u$ be a vector of unit norm and  $\alpha $ be the Rayleigh
quotient of $u$ with respect to a Hermitian matrix $B$. Let $s := u+
\tau (I-uu^\ast )t$, where $\tau\neq 0$ is chosen so that the
Rayleigh quotient $\rho(s)$  of $s$ is minimum over ${\rm span}
\{u,~(I-uu^\ast ) t\}.$ Write $J_{u,s} := (I-uu^\ast ) (B-\rho(s)I)
(I-uu^\ast ).$ Then the following relations hold:
\begin{equation}\label{equn16r}\relax
\tau = -\frac{\langle (B-\alpha I)u,(I-uu^\ast )t \rangle}{\langle
J_{u,s}(I-uu^\ast )t,(I-uu^\ast )t \rangle},
\end{equation}
\begin{equation}\label{equn17r}\relax
\rho(s) = \alpha - \frac{|\langle (B-\alpha I)u,(I-uu^\ast )t
\rangle |^2 }{\langle J_{u,s}(I-uu^\ast )t,(I-uu^\ast )t \rangle},
\end{equation}
\begin{equation}\label{equn18r}\relax
\langle J_{u,s}(u-s),u-s \rangle = \rho(u) - \rho(s),
\end{equation}
\begin{equation}\label{equn19r}\relax
(B-\rho(s) I)s = (B-\alpha I)u+J_{u,s}(s-u).
\end{equation}
\end{lem}
%Now, the following theorem applies the Lemma-\ref{thm1b} to the matrix $B:= V^*(A-\theta I)^*(A-\theta I)V,$ and the vectors $u:=y,$~$t:=z.$ It also derives a relation between the residual norms in the Rayleigh-Ritz projection and the LLS method.
\begin{thm}\label{thm1bnew}\relax
Let $Vy$ be the Ritz vector corresponding to a Ritz value $\theta$
 and the vector $V(I-yy^*)z$ be a
solution of the minimization problem (\ref{E2}). Let $\tau$ be a
scalar such that $s := y+ \tau (I-yy^\ast )z$ minimizes
$\displaystyle{\frac{\|(A-\theta I)Vs\|^2}{\|s\|^2}}$
over ${\rm span}\{y,\: (I-yy^\ast )z\}.$ Write
\begin{equation}\label{eq1}\relax
J_{y,s} := (I-yy^\ast )\Big(V^\ast (A-\theta I)^\ast (A-\theta I)V-\frac{\|(A-\theta I)Vs\|^2}{\| s \|^2} I \Big) (I-yy^\ast ).
\end{equation}
Then the following relations hold:
\begin{eqnarray}\label{E7}\relax
\tau = -\frac{\big\langle \big (V^\ast (A-\theta I)^\ast (A-\theta
I)V-\|A-\theta I)Vy\|^2 I \big )y,(I-yy^\ast )z
\big\rangle}{\big\langle J_{y,s}(I-yy^\ast )z,(I-yy^\ast )z
\big\rangle},
\end{eqnarray}
\begin{equation}\label{E8}\relax
\frac{\|(A-\theta I)Vs\|^2}{\| s \|^2} =\|(A-\theta I)Vy\|^2 - \tau
\|(A-\theta I)V(I-yy^\ast )z \|^2.
\end{equation}
\end{thm}
 \begin{proof}
Conveying the equation~(\ref{equn16r}) in the Lemma-\ref{thm1b} for the matrix $B:= V^*(A-\theta I)^*(A-\theta I)V$ and the vectors $u:=y,~t:=z$ gives the relation in the equation (\ref{E7}). Observe that in the Lemma-\ref{thm1b}, $\alpha=\|(A-\theta I)Vy\|^2$  for $B:= V^*(A-\theta I)^*(A-\theta I)V$ and $u:=y.~$  

Next, to prove the equation (\ref{E8}), carry the equations~\ref{equn17r}, \ref{equn18r} and \ref{equn19r} in the Lemma-1 for the matrix $B:= V^*(A-\theta I)^*(A-\theta I)V$ and the vectors $u:=y,~t:=z.$  This gives the following equations respectively:
\begin{eqnarray}\label{E9}\relax
 \frac{\|(A-\theta I)Vs\|^2}{\| s \|^2} & = & \|(A-\theta I)Vy\|^2  \nonumber \\
 & - &\frac{\big|\big\langle \big (V^\ast (A-\theta I)^\ast (A-\theta I)V-\|A-\theta I)Vy\|^2 I \big )y,(I-yy^\ast )z \big\rangle \big|^2 }{\big\langle J_{y,s}(I-yy^\ast )z,(I-yy^\ast )z \big\rangle},
\end{eqnarray}
\begin{eqnarray}\label{E10}\relax
\nonumber \langle J_{y,s}(y-s),y-s \rangle & = &\tau^2 \Big(\|(A-\theta I)V(I-yy^\ast )z\|^2 -\frac{\|(A-\theta I)Vs\|^2}{\| s \|^2}\cdot\|(I-yy^\ast )z\|^2\Big) \\
& = & \|(A-\theta I)Vy\|^2-\frac{\|(A-\theta I)Vs\|^2}{\| s \|^2},
\end{eqnarray}
\begin{multline}\label{E11}\relax
 \Big(V^*(A-\theta I)^*(A-\theta I)V- \frac{\|(A-\theta I)Vs\|^2}{\| s \|^2} I\Big) s  \\
= \big(V^*(A-\theta I)^*(A-\theta I)V-\|(A-\theta I)Vy\|^2 I\big)
y+J_{y,s}(s-y).
\end{multline}

Since $y\perp (I-yy^\ast )z,$ by using the equation~(\ref{E6}), the equations (\ref{E7}) and (\ref{E9}) gives the following relations:
\begin{equation}\label{E12}\relax
\frac{\tau}{\|(A-\theta I)V(I-yy^\ast )z\|^2}= \frac{1}{\langle
J_{y,s}(I-yy^\ast )z,(I-yy^\ast )z \rangle},
\end{equation}
\begin{equation*}\label{E13}\relax
\frac{\|(A-\theta I)Vs\|^2}{\| s \|^2} =\|(A-\theta I)Vy\|^2 -
\frac{\|(A-\theta I)V(I-yy^\ast )z \|^4 }{\langle J_{y,s}(I-yy^\ast
)z,(I-yy^\ast )z \rangle}.
\end{equation*}
Now, the equation (\ref{E8}) follows from the last two equations.
\end{proof}

The equation (\ref{E8}) gives a relation between residual norms in the Rayleigh-Ritz projection and LLS methods. The following theorem derives a few more  relations by utilizing the equation (\ref{E8}) .% between them which will be helpful later.
%residual norms in the Rayleigh-Ritz and the LLS methods.

\begin{thm}\label{Vs1}\relax
Let the vector $V(I-yy^*)z$ be a
solution of the minimization problem (\ref{E2}). Let $\tau$ be a
scalar such that $s := y+ \tau (I-yy^\ast )z$ minimizes
$\displaystyle{\frac{\|(A-\theta I)Vs\|^2}{\|s\|^2}}$ over ${\rm span}\{y,\: (I-yy^\ast )z\}.$ Then, the following equations hold true:
\begin{equation}\label{E14}\relax
\frac{\|(A-\theta I)Vs\|^2}{\|s\|^2}=\Big(\frac{\tau-1}{\tau}\Big)\frac{\|(A-\theta I)V(I-yy^*)z)\|^2}{\|(I-yy^*)z\|^2},
\end{equation}
and
\begin{equation}\label{E15}\relax 
(\tau-1)\Big(\|(A-\theta I)Vy\|^2-\frac{\|(A-\theta I)Vs\|^2}{\|s\|^2} \Big)= \|(A-\theta I)Vs\|^2\Big(1- \frac{1}{\|s\|^2} \Big).    
\end{equation}
\end{thm}
\begin{proof}
As the vector $(I-yy^\ast )z$ is a solution of the minimization problem~(\ref{E2}), it satisfies the equation (\ref{E6}). Now, on expanding the expression $\|(A-\theta I)Vs\|^2$ by using $s=y+\tau (I- yy^\ast )z,$ we have
\begin{equation}\label{E16}\relax
\|(A-\theta I)Vs\|^2=\|(A-\theta I)Vy\|^2+(\tau^2-2\tau)
\|(A-\theta I)V(I-yy^\ast )z \|^2.
\end{equation}
As $s=y+\tau(I-yy^\ast)z$ and $\|y\|=1,$ we have $\|s\|^2 =1+\tau^2 \|(I-yy^\ast)z\|^2.$ This implies $$\|(A-\theta I)Vs\|^2=\frac{\|(A-\theta I)Vs\|^2}{\|s\|^2}.(1+\tau^2  \| (I-yy^\ast)z\|^2).$$ Now, substitute the equations (\ref{E8}) and (\ref{E16}) in the above equation. On simplification, this gives the following relation:
$$(\tau-1)\frac{\|(A-\theta I)V(I-yy^\ast )z\|^2}{\|(I-yy^\ast )z\|^2} = \tau(\|(A-\theta I)Vy\|^2-\tau \|(A-\theta I)V(I-yy^*)z\|^2).$$
Recall from the equation (\ref{E8}) that the right-hand side of the above equation is equal to $\tau \frac{\|(A-\theta I)Vs\|^2}{\|s\|^2}.$ Thus, we have
$$ \frac{\|(A-\theta I)Vs\|^2}{\|s\|^2} =\Big(\frac{\tau-1}{\tau}\Big)\frac{\|(A-\theta I)V(I-yy^\ast )z\|^2}{\|(I-yy^\ast )z\|^2}.$$
Therefore, we proved the equation (\ref{E14}). To prove the equation (\ref{E15}), observe the following from the equations (\ref{E4}), (\ref{E8}) and (\ref{E16}): 
\begin{equation}\label{E17}\relax
(\tau-1) \|(A-\theta
I)Vy\|^2=\tau  \|(A-\theta I)Vm\|^2-\frac{\|(A-\theta I)Vs\|^2}{\|Vs\|^2},
\end{equation}
and
\begin{equation}\label{E18}\relax
(\tau-1) \frac{\|(A-\theta
I)Vs\|^2}{\|Vs\|^2}=\tau  \|(A-\theta I)Vm\|^2-\|(A-\theta I)Vs\|^2.
\end{equation}
Now, subtracting one of the above equation from the other gives the equation (\ref{E15}).
\end{proof}

In the Theorems-\ref{thm1bnew} and \ref{Vs1}, we have seen that the relations between norms of residuals in the LLS method involve the scalar $\tau.$ The following theorem gives a lower bound for the scalar $\tau$.% in the LLS method. 
\begin{thm}\label{thm3}\relax
Let $\tau$ be a scalar the same as that in the Theorem-\ref{thm1bnew}. Then $1 \leq \tau.$
\end{thm}
\begin{proof}
%From equation~(\ref{equn3}), we have $\|(A-\theta I)Vm\|^2 \leq \|(A-\theta I)Vy\|^2.$
By noting that $\|Vm\|^2 = \|Vy\|^2+\|V(I-yy^\ast)z\|^2
=1+\|(I-yy^\ast)z\|^2 \geq 1,$ we have
\begin{equation*}\label{E19}\relax
\frac{\|(A-\theta I)Vm\|^2}{\|Vm\|^2} \leq \|(A-\theta I)Vm\|^2.
\end{equation*}
Now, from equations~(\ref{E3}) and (\ref{E8}), we have
$$\|(A-\theta I)Vy\|^2 - \tau \|(A-\theta I)V(I-yy^\ast )z \|^2=\frac{\|(A-\theta I)Vs\|^2}{\| s \|^2} \leq \frac{\|(A-\theta I)Vm\|^2}{\| m \|^2}.$$
Now $\tau \geq 1$ is followed from equation~(\ref{E4}) and the above.
\end{proof}

The previous theorem has shown that $\tau \geq 1.$ By using the equation (\ref{E15}), observe that if $\tau=1$ then either $\|s\|^2=1$ or $\|(A-\theta I)Vs\|^2=0.$ That means, when $\tau=1,$ either $\|(I-yy^\ast)z\|=0$ or $(\theta,Vs)$ is an exact eigenpair of $A.$ Hence, in what follows we assumed that $\tau \neq 1.$  

\subsection{Comparison of line search least squares with refined projection}
In the previous section, we compared the residual norms in the Rayleigh-Ritz projection and Line search Least squares(LLS) methods. In this subsection, we establish a connection between the LLS and refined projection methods. 

Recall that an  approximate  eigenvalue $\theta$ in the refined projection method is the same as that in the Rayleigh-Ritz projection, and $u_R:=Vz_R$ is an eigenvector approximation, where $z_R$ is a right singular vector corresponding to the smallest non-zero singular value $\sigma^2$ of a matrix $(A-\theta I)V.$ Hence, a vector $z_R$ satisfies the following relations:
\begin{equation}\label{ref}\relax
V^\ast (A-\theta I)^\ast (A-\theta I)Vz_R =
\sigma^2z_R~~\mbox{and}~~\|(A-\theta I)Vz_R\|^2 = \sigma^2.
\end{equation}

%Recall that in the refined projection method, eigenvalue approximation $\theta$ is the same as that in the Rayleigh-Ritz projection, but an eigenvector approximation is $u_R=Vz_R,$ where
%$z_R$ is a right singular vector corresponding to the smallest non-zero singular value $\sigma^2$ of a matrix $(A-\theta I)V.$ Thus, a refined Ritz vector, an eigenvector approximation in the refined projection method
%% corresponding to Ritz value $\theta$ 
%satisfies the relation:
%\begin{equation}\label{ref}\relax
%V^\ast (A-\theta I)^\ast (A-\theta I)Vz_R =
%\sigma^2z_R~~\mbox{and}~~\|(A-\theta I)Vz_R\|^2 = \sigma^2.
%\end{equation}

Using the normal equations (\ref{E5}) of a least
squares problem~(\ref{E2}), observe that a vector $m= y+(I-yy^\ast )z$ satisfies
the following equation:
\begin{equation}\label{s1}\relax
V^\ast (A-\theta I)^\ast (A-\theta I)V\big(y+(I-yy^\ast)z\big) =
Ky~~\mbox {where} ~~K= \|(A-\theta I)Vm\|^2.
\end{equation}
Now, take an inner product on both sides with a vector $z_R$ and use the 
equation~(\ref{ref}) to obtain the following:
\begin{equation}\label{s2}\relax
\sigma^2 z_R^\ast\big(y+(I-yy^\ast )z\big) = Kz_R^\ast y \Rightarrow
z_R^\ast (I-yy^\ast)z = \frac{K-\sigma^2}{\sigma^2}z_R^\ast y.
\end{equation}
The above equation shows that the ratio of $z_R^\ast y$ to $z_R^\ast
(I-yy^\ast)z$ is real. In fact, the ratio is positive since $K = \|(A-
\theta I)Vm\|^2 \geq \frac{\|(A-\theta I)Vm\|^2}{\|m\|^2} \geq
\sigma^2.$ The last inequality follows since the refined Ritz
vector has smallest residual norm overall unit vectors in
the vector space spanned by column vectors of $V.$ 

By using the equation (\ref{s2}) and $s=y+\tau (I-yy^\ast)z$, we have
\begin{equation}\label{sn1}\relax
z_R^\ast s = z_R^\ast \big(y+\tau (I-yy^\ast)z\big) = z_R^\ast y
\Big(1+\tau\frac{K-\sigma^2}{\sigma^2}\Big).
\end{equation}
Since $\tau >1$ and $K \geq \sigma^2,$ the ratio of $z_R^\ast s$ to $z_R^\ast y$ is positive. Now, we restate this discussion in the form of a lemma for the future use.
\begin{lem}\label{lem2}\relax
Let $s,(I-yy^\ast)z$ be the same as that in the
equation~(\ref{E3}), and $Vz_R$ be the refined Ritz vector corresponding to the Ritz value $\theta.$  Then, $\frac{z_R^\ast s}{z_R^\ast
y}$ and  $\frac{z_R^\ast (I-yy^\ast)z }{z_R^\ast y}$ are positive.
\end{lem}

The above lemma inherently assumed that $z_R^\ast y \neq 0,$ which means the Ritz and refined Ritz vectors are not orthogonal. In numerical experiments, this statement holds true, in general. In the next theorem, we will use the
above lemma to derive a lower bound for
%Note that, in the above lemma, we inherently assumed that $z_R^\ast
%y \neq 0.$ This means, Ritz vector and refined Ritz vectors must not
%to be orthogonal to each other. It generally happens in many
%numerical experiments. However, theoretically, it is possible to
%show such cases should occur, specifically when solving for
%non-symmetric eigenvalue problems. In the next theorem, we are using
%above lemma to derive a lower bound for
$\frac{K-\sigma^2}{\sigma^2}.$

\begin{thm}\label{thms2}\relax
Let $\tau$ and a vector $(I-yy^\ast )z$ be the same as that in the Theorem-\ref{thm1bnew}. Assume that $\|(A-\theta I)Vz_R\|^2 = \sigma^2,$ $K= \|(A-\theta I)Vm\|^2,$ where $Vz_R$ is a refined Ritz
vector corresponding to the Ritz value $\theta,$  and  $m$ is a
solution vector of a least squares problem~(\ref{E2}). Then
\begin{equation}\label{s3}\relax
\frac{K-\sigma^2}{\sigma^2}  > \tau \|(I-y y^\ast)z\|^2.
\end{equation}
\end{thm}
\begin{proof}
Recall the equation (\ref{eq1}) from the previous subsection:
$$J_{y,s} := (I-yy^\ast )\Big(V^\ast (A-\theta I)^\ast
(A-\theta I)V-\frac{\|(A-\theta I)Vs\|^2}{\| s \|^2} I \Big)
(I-yy^\ast ).$$ Note that $s-y = \tau(I-yy^\ast )z.$ Then, by using the equations (\ref{E6}) and (\ref{ref}), we have
$$z_R^\ast J_{y,s}(s-y) =\tau(\sigma^2z_R^\ast (I-yy^*)z +\|(A-\theta I)V(I-yy^*)z\|^2z_R^\ast y)-\tau\frac{\|(A-\theta I)Vs\|^2}{\|s\|^2}z_R^\ast (I-yy^\ast)z.$$ By using the equations
(\ref{E4}), (\ref{s1}), and (\ref{s2}), this gives $$z_R^\ast J_{y,s}(s-y)=\tau(-\sigma^2+\|(A-\theta I)Vy\|^2)z_R^\ast y-\tau\frac{\|(A-\theta I)Vs\|^2}{\|s\|^2}z_R^\ast (I-yy^\ast)z.$$ 
Recall the following equation (\ref{E11}) from the previous subsection:
$$\Big(V^*(A-\theta I)^*(A-\theta I)V- \frac{\|(A-\theta I)Vs\|^2}{\| s \|^2} I\Big) s  \\
= \big(V^*(A-\theta I)^*(A-\theta I)V-\|(A-\theta I)Vy\|^2 I\big)
y+J_{y,s}(s-y).$$
Apply an inner product on both sides of the above equation with a vector
$z_R.$ In the resulting equation substitute $z_R^\ast J_{y,s}(s-y)$ from the previous equation in the right-hand side expression. Then use the equation ~(\ref{ref}) on the left-hand side expression of the same equation. It  gives the following:
%, and   substituting the previous equation for  by using the equation,  recalled at the beginning of the proof  gives the following:
\begin{multline}\label{s4}\relax
\Big(\sigma^2-\frac{\|(A-\theta I)Vs\|^2}{\|s\|^2}\Big) z_R^\ast s =
(1-\tau)\big(\sigma^2-\|(A-\theta I)Vy\|^2 \big)z_R^\ast
y-\tau\frac{\|(A-\theta I)Vs\|^2}{\|s\|^2}z_R^\ast (I-yy^\ast)z.
\end{multline}
Now, divide the both sides of the above equation with $z_R^\ast y$
to obtain the following:
\begin{equation}\label{sn2}\relax
\Big(\sigma^2-\frac{\|(A-\theta I)Vs\|^2}{\|s\|^2}\Big)
\frac{z_R^\ast s}{z_R^\ast y} = (1-\tau)\big(\sigma^2-\|(A-\theta
I)Vy\|^2 \big)-\tau\frac{\|(A-\theta
I)Vs\|^2}{\|s\|^2}\frac{z_R^\ast (I-yy^\ast)z}{z_R^\ast y}.
\end{equation}
Recall from ~Lemma-\ref{lem2} that $\frac{z_R^\ast s}{z_R^\ast y}$
and $\frac{z_R^\ast (I-yy^\ast)z}{z_R^\ast y}$ are positive, and from the Theorem-\ref{thm3} that $\tau \geq 1.$ By using these, the following inequality relation follows from the equations (\ref{s2}) and (\ref{sn1}).
$$\frac{z_R^\ast s}{z_R^\ast y} > \frac{z_R^\ast (I-yy^\ast)z}{z_R^\ast y}.$$
As $\sigma^2-\frac{\|(A-\theta I)Vs\|^2}{\|s\|^2}$ is non-positive, by using
the above inequation, the equation (\ref{sn2}) gives 
$$(1-\tau)\big(\sigma^2-\|(A-\theta I)Vy\|^2 \big)-\tau\frac{\|(A-\theta I)Vs\|^2}{\|s\|^2}\frac{z_R^\ast (I-yy^\ast)z}{z_R^\ast y} < \Big(\sigma^2-\frac{\|(A-\theta I)Vs\|^2}{\|s\|^2}\Big) \frac{z_R^\ast (I-yy^\ast)z}{z_R^\ast y}.$$
Now, by rearranging the terms, the above inequation can be written as follows:
$$(\tau-1)(\|(A-\theta I)Vy\|^2-\sigma^2) < \Big((\tau-1)\frac{\|(A-\theta I)Vs\|^2}{\|s\|^2}+\sigma^2\Big)\frac{z_R^\ast (I-yy^\ast)z}{z_R^\ast y}.$$
%Again by using 
As $\sigma^2 \leq \frac{\|(A-\theta I)Vs\|^2}{\|s\|^2},$ we have $\Big(\|(A-\theta I)Vy\|^2-\frac{\|(A-\theta
I)Vs\|^2}{\|s\|^2}\Big) \leq (\|(A-\theta I)Vy\|^2-\sigma^2).$ Since $\tau > 1,$ by using  these two inequalities the above equation gives the following relation:
$$(\tau-1)\Big(\|(A-\theta I)Vy\|^2-\frac{\|(A-\theta
I)Vs\|^2}{\|s\|^2}\Big) < \tau\frac{\|(A-\theta I)Vs\|^2}{\|s\|^2} .\frac{z_R^\ast (I-yy^\ast)z}{z_R^\ast y}.$$
Now, on substituting the equation~(\ref{s2}) this inequality gives 
$$\frac{(\tau-1)\Big(\|(A-\theta I)Vy\|^2-\frac{\|(A-\theta
I)Vs\|^2}{\|s\|^2}\Big)}{\tau\frac{\|(A-\theta I)Vs\|^2}{\|s\|^2}}
<  \frac{K-\sigma^2}{\sigma^2}.$$
Further, by using the equation (\ref{E15}) the left-hand side expression in the above equation becomes equal to $\frac{\|s\|^2-1}{\tau}.$ Thus, we have
$$\frac{\|s\|^2-1}{\tau} <  \frac{K-\sigma^2}{\sigma^2}.$$
As $s=y+\tau (I-yy^\ast)z$ and $\|y\|=1,$ we have $\|s\|^2=1+\tau^2\|(I-yy^\ast)z\|^2.$ Therefore, on substituting
this in the above equation, we get the required inequality as in the equation (\ref{s3}).
\end{proof}

The above theorem gives a lower bound for $\frac{K-\sigma^2}{\sigma^2}.$ In order to derive an upper bound for $\frac{K-\sigma^2}{\sigma^2},$ we define the following function of a variable $\alpha:$ 
%for $0 \leq \alpha < \tau$ 
\begin{equation}\label{s6}\relax
f(\alpha) = (\tau-1)\Big(\frac{\|(A-\theta
I)Vs\|^2}{\|s\|^2}-\|(A-\theta
I)Vy\|^2\Big)+(\tau-\alpha)\frac{\|(A-\theta I)Vs\|^2}{\|s\|^2}
\frac{K-\sigma^2}{\sigma^2}.
\end{equation}
Now, the following lemma describes the characteristics of the function $f(\alpha).$
\begin{lem}\label{VS-lem1}\relax
Let $f(\alpha)$ be a function of $\alpha$ defined as in the equation (\ref{s6}). Then the following are hold true:\\
a) $f(\alpha)$ is a monotonic decreasing function of $\alpha.$\\
b) If $f(\alpha) \leq 0$ for any  $0 \leq \alpha  < \tau$ then $\alpha$ satisfies
the inequation: $$ \tau\|(I-yy^\ast)z\|^2 < \frac{\|s\|^2-1}{(\tau-\alpha)}.$$
c) There exists a root between $0$ and $\tau$ for the equation $f(\alpha)=0.$
\end{lem}
\begin{proof}
a) For the given function $f(\alpha),$ we have
$$f'(\alpha) =  -\frac{\|(A-\theta I)Vs\|^2}{\|s\|^2}
\Big(\frac{K-\sigma^2}{\sigma^2}\Big).$$
To see $f'(\alpha) \leq 0,~\forall \alpha,$ use the optimal property of residual norms in the refined projection method and observe $\sigma^2 \leq K.$ Therefore, $f(\alpha)$ is a monotonically decreasing function of $\alpha.$\\
b) The substitution of the equation (\ref{E15}) in the equation (\ref{s6}) leads to
\begin{equation}\label{eq-VS3}\relax
f(\alpha) = \Bigg((1-\|s\|^2)+(\tau-\alpha)\Big( \frac{K-\sigma^2}{\sigma^2}\Big)\Bigg)\frac{\|(A-\theta I)Vs\|^2}{\|s\|^2}.
\end{equation}
Thus, $f(\alpha) \leq 0$ implies
$(1-\|s\|^2)+(\tau-\alpha)(
\frac{K-\sigma^2}{\sigma^2})$ is non-positive. Therefore, by using the equation (\ref{s3}), this gives the required inequality as  $0 \leq \alpha < \tau.$\\
c) We have
$$f(0)=  \Bigg((1-\|s\|^2)+\tau\Big(\frac{K-\sigma^2}{\sigma^2}\Big)\Bigg)\frac{\|(A-\theta I)Vs\|^2}{\|s\|^2}.$$
By using $\|s\|^2=1+\tau^2\|(I-yy^\ast)z\|^2,$ the above equation can be written as
$$f(0) = \tau\Bigg(-\tau \|(I-yy^\ast)z\|^2+\Big(\frac{K-\sigma^2}{\sigma^2}\Big)\Bigg)\frac{\|(A-\theta I)Vs\|^2}{\|s\|^2}.$$
As $\tau >1$, it is an easy to see that $f(0) >0$ by using (\ref{s3}) and the above equation.\\
Similarly, consider
$$f(\tau)=(1-\|s\|^2)\frac{\|(A-\theta I)Vs\|^2}{\|s\|^2}. $$
As $\|s\|^2=1+\tau^2\|(I-yy^\ast)z\|^2 >1,$ we have $f(\tau) <0.$ Therefore, we have $f(0) >0$ and $f(\tau) <0.$ Hence, there exists a root for the equation $f(\alpha)=0$ between $0$ and $\tau.$
\end{proof}

The Lemma-\ref{VS-lem1} has shown that equation $f(\alpha)=0$ has a solution  between $0$ and $\tau.$ Let $\alpha_3$ is such a root. Then, from the equation (\ref{eq-VS3}) we have
\begin{equation}\label{eq-VS4}\relax
\frac{K-\sigma^2}{\sigma^2} =\frac{\|s\|^2-1}{\tau-\alpha_3}.
\end{equation}
Note that the above equation turns the problem of finding an upper bound for $\frac{K-\sigma^2}{\sigma^2}$ into deriving an upper bound for $\alpha_3.$   To derive an upper bound for $\alpha_3,$ which depends only on the scalar $\tau,$ we make use of the following function of $\alpha:$
%\begin{equation}\label{KS-H}\relax
%H(\alpha)=(\tau-\alpha)f(\alpha)-(\tau-\alpha)^2g(\alpha)+(\tau-\alpha)g(\alpha), 
%\end{equation}
%where 
\begin{equation}\label{eq-VSAT1}\relax
g(\alpha)= \Big(\frac{\|s\|^2-1}{\tau-\alpha}-\tau \|(I-yy^\ast)z\|^2\Big)\frac{\|(A-\theta I)Vs\|^2}{\|s\|^2}.   
\end{equation}

Note that the function $g(\alpha)$ is obtained by multiplying the difference between both sides of the inequality in the Lemma-\ref{VS-lem1}(b) with $\frac{\|(A-\theta I)Vs\|^2}{\|s\|^2}.$ In the following lemma,  we characterize the function $g(\alpha)$ defined in the above equation and will establish its relation with the function $f(\alpha).$

\begin{lem}\label{OBV}\relax
Let $g(\alpha)$ be a function defined as in the equation (\ref{eq-VSAT1}). Then, 
%a). $$g(\tau-1)=(\tau-1) \tau \|(I-yy^\ast)z\|^2 \frac{\|(A-\theta I)Vs\|^2}{\|s\|^2}.$$ 
 the functions $g(\alpha)$ and $g(\alpha)-f(\alpha)$ are monotonically increasing functions of $\alpha$ in the interval $[0,\tau).$\\
\end{lem}
\begin{proof}
%Substituting $\|s\|^2=1+\tau^2 \|(I-yy^\ast)z\|^2$ the equation (\ref{eq-VSAT1}) proves the part-a. 
Use $\|s\|^2=1+\tau^2 \|(I-yy^\ast)z\|^2$  in the equation (\ref{eq-VSAT1})to observe that
\begin{equation}\label{OBV4}\relax
g(\alpha)= \Big(\frac{\tau}{\tau-\alpha}-1\Big)\tau \|(I-yy^\ast)z\|^2 \frac{\|(A-\theta I)Vs\|^2}{\|s\|^2}.
\end{equation}
This shows that
$$g'(\alpha)= \frac{\tau^2 \|(I-yy^\ast)z\|^2}{(\tau-\alpha)^2}\frac{\|(A-\theta I)Vs\|^2}{\|s\|^2} >0.$$
Thus, the  proof is over as $f'(\alpha) \leq 0, ~\forall \alpha$ from the Lemma-\ref{VS-lem1}(a). 
\end{proof}

In what follows, with the help of the function $g(\alpha)$ we derive an upper bound for $\alpha_3,$ a solution of the equation $f(\alpha)=0.$  
For this, the following theorem introduce $\alpha_6,$ a root of the equation $f(\alpha)-\tau \alpha  \|(I-yy^\ast)z\|^2 \frac{\|(A-\theta I)Vs\|^2}{\|s\|^2}=0$ and determine a relation between $\alpha_6$ and $\alpha_3.$ 
\begin{thm}\label{VSLNSMKV1}\relax
Let $\alpha_6 \neq \tau$ is such that $f(\alpha_6)-\tau \alpha_6\|(I-yy^\ast)z\|^2 \frac{\|(A-\theta I)Vs\|^2}{\|s\|^2}=0,$ then $2\alpha_3-\tau \leq \alpha_6 < \alpha_3$ and
\begin{equation}\label{SV14}\relax
\alpha_3=\frac{2\tau \alpha_6}{\tau+\alpha_6},
% \tau-\alpha_6=\frac{2\tau(\tau-\alpha_3)}{(2\tau-\alpha_3)},
\end{equation}
where $\alpha_3 < \tau$ is same as in the Lemma-\ref{VS-lem1}. 
\end{thm}
\begin{proof}
Note that substituting the equation (\ref{eq-VS4}) in the equation (\ref{eq-VS3}) gives   
\begin{equation*}
f(\alpha) = \Bigg((1-\|s\|^2)+(\tau-\alpha)\Big( \frac{\|s\|^2-1}{\tau-\alpha_3}\Big)\Bigg)\frac{\|(A-\theta I)Vs\|^2}{\|s\|^2}.
\end{equation*}
Further, using  $\|s\|^2=1+\tau^2 \|(I-yy^\ast)z\|^2$ the above equation can be simplified as the following:
\begin{equation}\label{eq-OBVnew}\relax
f(\alpha) = \Bigg(-\tau+ \frac{(\tau-\alpha)\tau}{\tau-\alpha_3}\Bigg)\tau \|(I-yy^\ast)z\|^2 .\frac{\|(A-\theta I)Vs\|^2}{\|s\|^2}.
\end{equation}
Let $\alpha_6 \neq \tau.$ Since $f(\alpha_6)-\tau \alpha_6\|(I-yy^\ast)z\|^2  \frac{\|(A-\theta I)Vs\|^2}{\|s\|^2}=0$ the equation (\ref{eq-OBVnew}) would imply
\begin{multline}
f(\alpha_6) = \Bigg(-\tau+ \frac{(\tau-\alpha_6)\tau}{\tau-\alpha_3}\Bigg)\tau \|(I-yy^\ast)z\|^2 .\frac{\|(A-\theta I)Vs\|^2}{\|s\|^2}\\=\tau \alpha_6\|(I-yy^\ast)z\|^2  \frac{\|(A-\theta I)Vs\|^2}{\|s\|^2}.~~~~~~~~~~~~~~~~~~~~~~~~~~~~~~~~~~~~~~~~~~
\end{multline}
As $\tau \|(I-yy^\ast)z\|^2 \frac{\|(A-\theta I)Vs\|^2}{\|s\|^2} \neq 0$ and $\tau \neq \alpha_6,$ on simplification, the above equation gives 
$$\frac{\tau+\alpha_6}{\tau-\alpha_6}=\frac{\tau}{\tau-\alpha_3}.$$ Using Componendo and Dividendo, the above equation proves
the equation (\ref{SV14}).  
\end{proof}

Recall that  our aim is to determine an upper bound for $\alpha_3$ in terms of $\tau.$ From the above theorem it is equivalent to identify an upper bound for $\alpha_6.$ For this,the following section introduces another scalar, called $\alpha_7$ and determine its location with respect to $\alpha_3$ and $\alpha_6$ on the real line.
\section{Sectional Formulae}
In this section, we introduce a scalar $\alpha_7 <\tau,$ a root of the equation $f(\alpha)-g(\alpha)=0.$
% and establish its relation with a scalar $\alpha_6$ in the previous section.
The following lemma determines a sufficient condition for $\alpha_6$ to divide $\alpha_7$ and $\alpha_3$ externally on the real line.
\begin{lem}\label{OBV-a7}\relax
Let $f(\alpha),g(\alpha)$ and $\alpha_6$ be the same as mentioned in the Lemma-\ref{OBV} and the Theorem-\ref{VSLNSMKV1}. Assume that $\alpha_7 <\tau$ is a root of the equation $f(\alpha)-g(\alpha)=0.$  Then $\alpha_7 \leq \alpha_3.$ Further, If $\alpha_7 \leq \tau-1$ then $\alpha_6 \leq \alpha_7.$ 
\end{lem}
\begin{proof}
Observe from the equation (\ref{OBV4}) that $$g(\tau-1)=(\tau-1)\tau\|(I-yy^\ast)z\|^2\frac{\|(A-\theta I)Vs\|^2}{\|s\|^2}$$ and 
$$g(\alpha_7)=\Big(\frac{\alpha_7}{\tau-\alpha_7}\Big)\tau\|(I-yy^\ast)z\|^2\frac{\|(A-\theta I)Vs\|^2}{\|s\|^2}. $$
Note that $\alpha_7 \leq \tau-1.$ Since $g(\alpha)$ is monotonically increasing function, from the Lemma--\ref{OBV} and using the first equation in the above, we have 
$$g(\alpha_7) \geq \tau\alpha_7 \|(I-yy^\ast)z\|^2\frac{\|(A-\theta I)Vs\|^2}{\|s\|^2}. $$
Using $f(\alpha_7)=g(\alpha_7)$ the above equation gives the following inequality:
$$f(\alpha_7) \geq \tau\alpha_7 \|(I-yy^\ast)z\|^2\frac{\|(A-\theta I)Vs\|^2}{\|s\|^2}. $$
Now, $\alpha_6 \leq \alpha_7$  is proved by invoke from the Lemma-\ref{VS-lem1}(a) that $f(\alpha)-\tau\alpha\|(I-yy^\ast)z\|^2\frac{\|(A-\theta I)Vs\|^2}{\|s\|^2}$ is a monotonically decreasing function of $\alpha$ and  
 %  Also, note from the Lemma that 
$$f(\alpha_6)-\tau\alpha_6 \|(I-yy^\ast)z\|^2\frac{\|(A-\theta I)Vs\|^2}{\|s\|^2}=0.$$ 
%Now, we prove $\alpha_6 \leq \alpha_7$ \ni From the equation we have
Note that $\alpha_7 \leq \alpha_3$ follows as $f(\alpha_3)=0$ and $f(\alpha)$ is a monotonically decreasing function of $\alpha$ from the Lemma-\ref{VS-lem1}(a). 
\end{proof}

The Lemma-\ref{OBV-a7} says that if $\alpha_7 \leq \tau-1$ then $\alpha_6 \leq \tau-1.$ This implies $\alpha_3 \leq \frac{2\tau(\tau-1)}{2\tau-1}$ as $\alpha_3= \frac{2\tau\alpha_6}{\tau+\alpha_6}$ from the Theorem-\ref{VSLNSMKV1}. By using the Harmonic mean and Arithemetic mean inequality this gives the following lemma:
\begin{lem}\label{OBV-a3}\relax
Let $\alpha_7$ and $\alpha_3$ be scalars the same as in the Lemma-\ref{OBV-Ek1}. If $\alpha_7 \leq \tau-1$ then $\alpha_3 \leq \tau-\frac{1}{2}.$
\end{lem}

The above lemma had given an upper bound for $\alpha_3$ when $\alpha_6$  externally divides $\alpha_7$ and $\alpha_3$ on the real line. In the following, we derive an upper bound for $\alpha_3$ when $\alpha_6$ lies in between $\alpha_7$ and $\alpha_3$ on the real line.
Further, In what follows we use the notation
$$z:= \tau\|(I-yy^\ast)z\|^2 \frac{\|(A-\theta I)Vs\|^2}{\|s\|^2}$$ for the convenience.

\begin{lem}\label{OBV-SB}\relax
Let $\alpha_7 \leq \alpha_6$ and  the point $C(\alpha_6,\alpha_6z)$ divides the points $A(\alpha_7,g(\alpha_7))$ and $B(\alpha_3,0)$ internally in the ratio $m:n.$ Then, note that the following relations hold true:
\begin{equation}\label{OBVSRM-sec}
\frac{m}{n}=\frac{g(\alpha_7)/z-\alpha_7}{\alpha_3}=\frac{\alpha_6-\alpha_7}{\alpha_3-\alpha_6}.
\end{equation}
\end{lem}
\begin{proof}
Note that $A,B,$ and $C$ are lies on the same straight line as $f(\alpha_3)=0,$ $\alpha_6z=f(\alpha_6)$ and  $f(\alpha_7)=g(\alpha_7)$ from the Theorem-\ref{VSLNSMKV1} and the Lemma-\ref{OBV-a7}. Further, as $C$ divides $A$ and $B$ internally in the ratio $m:n,$ we have
\begin{equation}\label{OBVSRM-MN}\relax
\Big(\frac{m\alpha_3+n\alpha_7}{m+n},\frac{ng(\alpha_7)}{m+n}\Big)=(\alpha_6,\alpha_6z).
\end{equation}
The above equation yields the following relations:
$$ \frac{m}{n} =\frac{\alpha_6-\alpha_7}{\alpha_3-\alpha_6}~\mbox{and}~\frac{m}{n}=\frac{g(\alpha_7)/z-\alpha_6}{\alpha_6} $$
The two relations in the above equation give:
$$\frac{m}{n}=\frac{g(\alpha_7)/z-\alpha_7}{\alpha_3}.$$
Hence the lemma proved.
\end{proof}

The Lemma-\ref{OBV-SB}  has found the ratio at which the point $C$ divides $A$ and $B.$ By using this the next theorem finds a relation between $\alpha_6,\alpha_7,$ and $ \tau.$
\begin{thm}\label{OBV-T67}\relax
Let $\alpha_3,\alpha_6,$ and $\alpha_7$ be the same as in the Lemma-\ref{OBV-SB}. Let $\tau\neq 1$ and $z =\tau \|(I-yy^\ast)z\|^2 \frac{\|(A-\theta I)Vs\|^2}{\|Vs\|^2} \neq 0.$ Then
\begin{equation}\label{SB-OBV-67}
\alpha_7=2\alpha_6-\tau.
\end{equation}
\end{thm}
\begin{proof}
From the equation (\ref{OBVSRM-sec}) note that
$$(\alpha_6-\alpha_7) =\Big(\frac{g(\alpha_7)/z-\alpha_7}{\alpha_3}\Big)(\alpha_3-\alpha_6).$$
Thus,
\begin{eqnarray}
 \tau+2\alpha_6 =\tau+\alpha_6+\alpha_7+(\alpha_6-\alpha_7)
=\tau+\alpha_6+\alpha_7+\Big(\frac{g(\alpha_7)/z-\alpha_7}{\alpha_3}\Big)(\alpha_3-\alpha_6). \nonumber
\end{eqnarray}
Recall from the equations (\ref{SV14}) and (\ref{eq-OBVnew}) that $f(2\alpha_6-\tau)=(\tau+2\alpha_6)z.$ Thus, as $f(\alpha_7)=g(\alpha_7)$ the above equation gives the relation:
\begin{eqnarray}\label{OBVSRM-67}\relax
f(2\alpha_6-\tau)=(\tau+\alpha_6+\alpha_7)z+
\Big(\frac{f(\alpha_7)/z-\alpha_7}{\alpha_3}\Big)(\alpha_3-\alpha_6)z\nonumber 
\\ =(\tau+\alpha_6)z+f(\alpha_7)-\Big(\frac{f(\alpha_7)/z-\alpha_7}{\alpha_3}\Big)\alpha_6 z.
\end{eqnarray}
Therefore, 
%the equation is proved.
\begin{equation}\label{OBV-Ek1}
f(2\alpha_6-\tau)-f(\alpha_7)=(\tau+\alpha_6)z-\Big(\frac{f(\alpha_7)/z-\alpha_7}{\alpha_3}\Big)\alpha_6 z.
\end{equation}
Now, by using the facts $(\tau+2\alpha_6)z =f(2\alpha_6-\tau)$ and $\alpha_3=2\tau\alpha_6/(\tau+\alpha_6)$ observe that 
\begin{eqnarray}
(\tau+\alpha_6)\alpha_3-\big(f(2\alpha_6-\tau)/z-(2\alpha_6-\tau)\big)\alpha_6=(\tau+\alpha_6)\alpha_3-2\tau=0.\nonumber
\end{eqnarray}
%In the above equation, we used the facts that  
Substituting this in the equation (\ref{OBV-Ek1}) gives the following relation:
$$f(2\alpha_6-\tau)-f(\alpha_7)=\Big(\frac{\big(f(2\alpha_6-\tau)/z-(2\alpha_6-\tau)\big)-\big(f(\alpha_7)/z-\alpha_7\big)}{\alpha_3}\Big)\alpha_6z.$$
This can be written as follows:
$$\Big(f(2\alpha_6-\tau)-f(\alpha_7)\Big)\Big(\frac{\alpha_3-\alpha_6}{\alpha_6}\Big)=\frac{\big(\alpha_7-(2\alpha_6-\tau)\big)\alpha_6z}{\alpha_3}.$$
As $\alpha_3$ is harmonic mean of $\alpha_6$ and $\tau$ we have $\frac{\alpha_3-\alpha_6}{\alpha_6}=\frac{\tau-\alpha_3}{\tau}.$ Substituting this in the above equation gives
$$\Big(f(2\alpha_6-\tau)-f(\alpha_7)\Big)\Big(\frac{\tau-\alpha_3}{\tau}\Big)=\frac{\big(\alpha_7-(2\alpha_6-\tau)\big)\alpha_6z}{\alpha_3}.$$
Now, we prove that if $2\alpha_6-\tau \neq \alpha_7$ then $\alpha_6=\alpha_3=\tau.$
Recall from the equation (\ref{eq-OBVnew}) that $f(\alpha)$ is a straightline in the variable $\alpha$ and its slope is $\frac{-\tau z}{\tau-\alpha_3}.$ Thus, $\frac{f(x_1)-f(x_2)}{x_1-x_2}$ is constant for any $x_1$ and $x_2.$ Using this observe that $\alpha_6=\alpha_3$ from the above equation. But this implies $\alpha_3=\tau$ as $\alpha_3=\frac{2\tau\alpha_6}{\tau+\alpha_6}.$ A contradiction to the assumption that $\alpha_3 \neq \tau.$ Note that if $\alpha_3=\tau$ then either $\tau=1$ or $(\theta,Vy)$ is an exact eigenpair of $A.$ Therefore $\alpha_7=2\alpha_6-\tau.$
\end{proof}

By using the above theorem the following lemma gives the values of $\alpha_6$ and $\alpha_3$ in terms of $\tau,$ when $\alpha_6$ stays in between $\alpha_7$ and $\alpha_3$ as mentioned in the Lemma-\ref{OBV-SB}.
\begin{lem}\label{VSB6T}\relax
Let $g(\alpha)$ and $f(\alpha)$ be functions of $\alpha$ defined as in the equations (\ref{OBV4}) and (\ref{eq-OBVnew}) respectively. Assume that scalars $\alpha_3,$ $\alpha_6,$ and $\alpha_7$ are the same as in the Lemma-\ref{OBV-SB}. Then
\begin{equation}\label{OBV-a3T}\relax
\alpha_3 \leq \tau-\frac{1}{14},
\end{equation}
if the point $(\alpha_6,f(\alpha_6))$ internally divides $(\alpha_7,f(\alpha_7))$ and $(\alpha_3,0).$
\end{lem}
\begin{proof}
By using the equation (\ref{OBV4}) we have
$$g(2\alpha_6-\tau)=\frac{(2\alpha_6-\tau)z}{2(\tau-\alpha_6)}.$$
From the proof of the previous theorem note that  $f(2\alpha_6-\tau)=(\tau+2\alpha_6)z.$ Since $\alpha_7=2\alpha_6-\tau$ and $f(\alpha_7)=g(\alpha_7)$ this implies
$$(\tau+2\alpha_6)z=f(2\alpha_6-\tau)= g(2\alpha_6-\tau)=\frac{(2\alpha_6-\tau)z}{2(\tau-\alpha_6)}.$$
As $z \neq 0,$ on simplifying this gives the following quadratic equation in  $\alpha_6:$
$$4\alpha_6^2-(2\tau-2)\alpha_6-2\tau^2-\tau=0 .$$
As $\alpha_6 \leq \tau,$ the above equation gives 
%the value of $\alpha_6$ in terms of $\tau.$
$$\alpha_6=\frac{(\tau-1)+\sqrt{(\tau-1)^2+8\tau^2+4\tau}}{4}.$$
Thus, by using $\alpha_3=\frac{2\tau \alpha_6}{\tau+\alpha_6}$ we have
$$\alpha_3= \frac{2\tau\big(\tau-1+\sqrt{9\tau^2+2\tau+1}\big)}{5\tau-1+\sqrt{9\tau^2+2\tau+1}}.$$
As $\tau > 1,$  the right-hand side expression in the above equation is less than $\tau-\frac{1}{14}.$ This can be seen by plotting the graph by using any software such as MATLAB or DESMOS online grapher etc.
\end{proof}

From the Lemmas-\ref{OBV-a3} and \ref{VSB6T} observe that the inequality $\alpha_3 \leq \tau-\frac{1}{14}$ holds true, irrespective of the position of $\alpha_6$ with respect to the scalars $\alpha_3$ and $\alpha_7.$ We use this in the next section to find a bound for the ratio of residual norms in the refined and Rayleigh-Ritz projection methods. 
%\section{OSB-MR}
\section{Main results}
Recall from the equation (\ref{eq-VS4}) that $\frac{K-\sigma^2}{\sigma^2}=\frac{\tau^2 \|(I-yy^\ast)z\|^2}{\tau-\alpha_3}.$ Here, we used the fact that $\|s\|^2=1+\tau^2 \|(I-yy^\ast)z\|^2.$ Then using $\alpha_3 \leq \tau-1/14$ we have 
$$\frac{K}{\sigma^2} \leq 14\tau^2\|(I-yy^\ast)z\|^2+1.$$
This together with the Theorem-\ref{thms2} gives the result that we state in the form of a lemma here.
\begin{lem}\label{flemsv}\relax
Let $\tau$ and a vector $(I-yy^\ast )z$ be the same as in the Theorem-\ref{thm1bnew}. Also assume that $Vz_R$ is a refined Ritz
vector corresponding to the Ritz value $\theta.$ Then
\begin{equation*}
1+\tau \|(I-y y^\ast)z\|^2 \leq \frac{K}{\sigma^2} \leq 14\tau^2\|(I-yy^\ast)z\|^2+1,
\end{equation*}
 where $\sigma^2 =\|(A-\theta
I)Vz_R\|^2,$ $K= \|(A-\theta I)Vm\|^2,$ and $m$ is a
solution vector of a least squares problem~(\ref{E2}).
\end{lem}

The Lemma-\ref{flemsv} has established the relation between residual norms in LLS and refined projection methods. It has shown that the residual norms in both the methods converge to zero together.
%if and only if the residual norms in the LLS method are converge to zero. 
Now, recall from the equations (\ref{E4}) and (\ref{E17}) that $$\|(A-\theta I)Vm\|^2 \leq \|(A-\theta I)Vy\|^2 \leq \frac{\tau \|(A-\theta I)Vm\|^2}{\tau-1}.$$ 

By using this equation, the Lemma-\ref{flemsv} gives the following main result which relates residual norms in Rayliegh-Ritz and refined projection methods.
\begin{thm}\label{mainSV}\relax
Let  $\theta$ be a Ritz value, and $Vy,~Vz_R$ be the corresponding eigenvector approximations in the Rayleigh-Ritz and refined projection methods respectively. Then
\begin{equation*}
 \frac{(\tau-1) \|(A-\theta I)Vy\|^2}{\tau (14\tau^2\|(I-yy^\ast)z\|^2+1)} \leq \sigma^2 \leq \frac{\|(A-\theta I)Vy\|^2}{1+\tau \|(I-y y^\ast)z\|^2 }.
\end{equation*}
\end{thm}
The above theorem relates residual norms in the Rayleigh-Ritz and refined projection methods. It helps us to predict the range of $\sigma^2,$ square of a  residual norm in the
refined projection method without computing a refined Ritz vector,
a right singular vector of $(A-\theta I)V.$ It just requires computing
 $\tau$ and $(I-yy^\ast )z.$ Note that $(I-yy^\ast)z$ is obtained by solving  normal equations for the
least squares problem in equation~(\ref{E2}). $\tau,$ a solution of the
problem considered in the equation~(\ref{E3}) is obtained by
solving an eigenvalue problem  of the following  matrix of order $2.$
$$B^\ast V^\ast (A-\theta I)^\ast (A-\theta I)VB,~~~\mbox{where}~~B= \Big[y~~ \frac{(I-yy^\ast)z}{\|(I-yy^\ast)z\|}\Big].$$
The above theorem may helps to create an efficient algorithm that use a
%for solving linear eigenvalue problems by using 
combination of refined projection and Rayleigh-Ritz projection methods for solving sparse linear eigenvalue problems.
\section{Numerical experiments}In this section, we demonstrate the theory developed so far. This section been divided into two subsections. The first part discusses a method to compute $\tau$ and $\|(I-yy^\ast)z\|.$ The second part reports numerical results.
\subsection{Implementation details}
In this section, we discuss how the LLS method obtains eigenvector approximations as the LLS method is the most recent one. The following theorem will be helpful to compute an eigenvector approximation in the Least squares and line search(LLS) method.
\begin{thm}\label{obs1}\relax
Let $(\theta,Vy)$ be a Ritz pair but not an exact eigenpair of $A$. Let the vector $x$ satisfies equation
\begin{equation}\label{eq17}\relax
V^\ast (A-\theta I)^\ast (A-\theta I)Vx = y.
\end{equation}
Then
$$\frac{V\big(y+K(I-yy^\ast )x\big)}{\|V\big(y+K(I-yy^\ast )x\big)\|}$$ 
is the corresponding eigenvector approximation in the least squares method,
where 
%$K$ satisfies equation
\begin{equation}\label{eq19}\relax
K = \|(A-\theta I)Vy\|^2 +K \langle (A-\theta I)V(I-yy^\ast )x,(A-\theta I)Vy \rangle.
\end{equation}
\end{thm}  
For the proof of the theorem; See Theorem-4 in \cite{Someshwaragift2}. Using the above theorem the LLS method computes a vector $\frac{(I-yy^\ast)z}{\|(I-yy^\ast)z\|}$ without computing $(I-yy^\ast)z$ explicitly. 

The LLS technique further improves an eigenvector approximation in the Theorem-\ref{obs1}  by using the Line-Search technique introduced in the Theorem-\ref{thm1bnew}. As mentioned in the previous section, LLS obtains it by solving the following eigenvalue problem of a matrix of order $2.$ 
$$B^\ast V^\ast (A-\theta I)^\ast (A-\theta I)VB,~~~\mbox{where}~~B= \Big[y~~ \frac{(I-yy^\ast)z}{\|(I-yy^\ast)z\|}\Big].$$ 

From the Lemma-\ref{flemsv} note that $\tau$ and $\|(I-yy^\ast)z\|^2$ are required to compare residual norms in the Rayleigh-Ritz projection, LLS, and refined projection methods. Observe that $\tau$ can be computed from the equation (\ref{E14}) since $\frac{\|(A-\theta I)Vs\|^2}{\|Vs\|^2}$ and $\frac{(I-yy^\ast)z}{\|(I-yy^\ast)z\|}$ are known.
% But, the LLS method does not compute $\tau$ and $(I-yy^\ast)z$ explicitly. However, I
The following lemma describes a procedure to compute $\|(I-yy^\ast)z\|^2.$

\begin{lem}\label{obs2}\relax
Let $s$ be 
an eigenvector of the matrix $B^\ast V^\ast (A-\theta I)^\ast (A-\theta I)VB,$ where $B= \Big[y~~ \frac{(I-yy^\ast)z}{\|(I-yy^\ast)z\|}\Big].$ If $\|s\|=1$ and $s=\frac{y+\tau(I-yy^\ast)z}{\sqrt{1+\tau^2 \|(I-yy^\ast)z\|^2}}$ then
$$\tau^2 \|(I-yy^\ast)z\|^2 = \frac{\|(I-yy^\ast)s\|^2}{1-\|(I-yy^\ast)s\|^2}.$$
%are same in the Theorem. 
\end{lem}
So far, In this section we discussed how to compute $\tau$ and $\|(I-yy^\ast)z\|.$ In the following subsection, we report the numerical results.
\subsection{Numerical results}
The numerical experiments have been conducted on many benchmark matrices from the Matrix Market Website. Here we report only two examples as all the experiments validated the theory in the previous sections. All the experiments have been conducted on Intel core i7 processor using MATLAB-R2016(b) with $eps=2.2204e-16.$ 
\begin{eg}\label{OBV-JD}\relax
In this example we used the Jacobi-Davidson method without restarting to compute right most
eigenvalues of the matrix $OLM5000.$ For details of the matrix; See Matrix Market Website. The initial vector has all its entries equal to $\frac{1}{\sqrt{n}},$ where $n$ is the order of the matrix. An eigenvector approximation in the LLS method is used in the correction equation. It solved approximately by using  20 iterations of un-restarted GMRES
method. In GMRES, we took the zero vector as an initial
approximation to the solution of the correction equation.

At each iteration of the Jacobi-Davidson method, we compute refined Ritz vector also and compare its residual norm with those in the Rayliegh-Ritz and LLS methods in accordance with the Lemma-\ref{flemsv}.
\end{eg}

In this example we fixed the size of a search subspace in the Jacobi-Davidson method to $200.$ It is well known that $\tau=1$ in the first iteration as the search subspace contains only initial vector.
\begin{figure}[!htb]
\begin{center}
\hfill \includegraphics[width = 5.5in,height=1.8in]{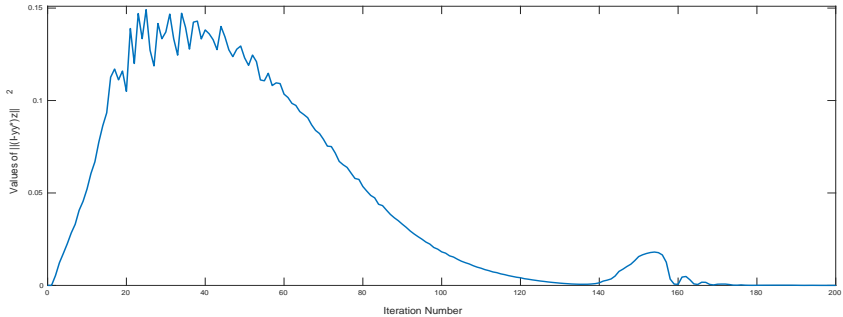} 
\end{center}\vs{-0.3cm}
\caption{Iteration numbers versus $\|(I-yy^\ast)z\|^2$}
\label{fig:1}
\end{figure}\vs{1cm}

\ni The Figure-\ref{fig:1} depicts the curves of 
%$\tau$ and 
$\|(I-yy^\ast)z\|^2$ against the iteration number. 
%The Figure- compares $K/\sigma^2$ with respect to its upper and lower bounds in the Lemma $1+14\|(I-yy^\ast)z\|^2$ and $1+\tau\|(I-yy^\ast)z\|^2$ respectively.
It is clear from the figure that as the iteration number grows, $\|(I-yy^\ast)z\|^2$ decreases. We found that from the $150^{th}$ iteration onwards its value is below $\mathcal{O}(10^{-5}).$ Thus, the Figure-\ref{fig:1} confirms the well known fact that near the convergence, eigenvector approximations in the refined method and Rayleigh-Ritz projection method almost coincide.
% The maximum value of $\|(I-yy^\ast)z\|^2$ observed is $0.15.$
\begin{figure}[!htb]
\begin{center}
\hfill \includegraphics[width = 5.5in,height=1.8in]{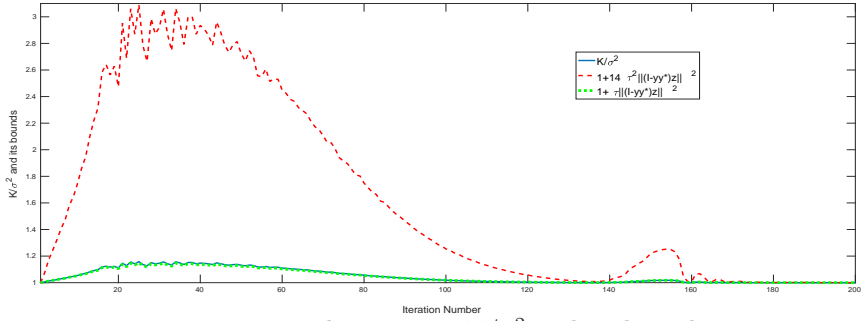} 
\end{center}\vs{-0.4cm}
\caption{Iteration numbers versus $K/\sigma^2$ and its bounds}
\label{fig:2}
\end{figure}

\ni The Figure-\ref{fig:2} shows $K/\sigma^2$ and its bounds in the Lemma-\ref{flemsv} against iteration number. Recall that $\sigma$ is residual norm in the refined projection method which is minimum over all unit vectors in the entire search subspace. From the figure it is easy to see that  $K/\sigma^2$ lies in the interval $(1,1.2),$ that means $\sigma^2 >(0.8)K.$  Here, $K=\|(A-\theta I)V(y+(I-yy^\ast)z\|^2,$ a residual norm of non-normalized vector $Vy+V(I-yy^\ast)z.$ Therefore, normalized residual norm of this vector will be much closer to residual norm in the refined projection.

\begin{figure}[!htb]
\begin{center}
\hfill \includegraphics[width = 5.5in,height=1.8in]{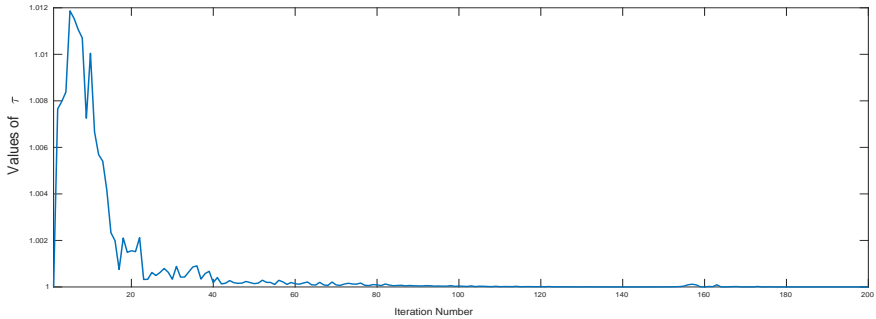} 
\end{center}\vs{-0.4cm}
\caption{Iteration numbers versus $\tau$}
\label{tau1f}
\end{figure}\vs{0.2cm}

\ni The Figure-\ref{tau1f} shows values of $\tau$ against iteration number. 
%It is found that $\tau=1$ after  
Observe from the figure  that  $\tau$ is in the interval $(1,1.013).$ However, we do not have any theoretical evidence on  a real number upper bound of $\tau.$ Thus, from this figure and the equations (\ref{E4}) and (\ref{E8}) it is evident that the line search technique brings only a marginal reduction in the residual norm obtained only with the least squares heuristics. \vs{0.1cm}

\begin{eg}\label{eg2}\relax
In this example the matrix is  $DW2048$ from the Matrix Market. We use the Arnoldi method with LLS to compute the eigenvalue with largest real  part. The initial vector chosen as $ones(2048,1).$ Since search subspace updataion in the Arnoldi method doesn't require eigenvector approximation like the Jacobi-Davidson method, we tested our theoretical results with explicitly restarting Arnoldi method. In restarting Arnoldi method the size of a Krylov subspace is fixed to $10$ for this example. However, the same scenario that we present here is observed with subspaces of larger size. 
At the end of each restart, an initial vector updated by an eigenvector approximation at hand obtained using the LLS method. 
% we compute $K$ by using the equation (3.10) in \cite{Someshwaragift2}.
\end{eg}

\begin{figure}[!htb]
\begin{center}
\hfill \includegraphics[width = 5.5in,height=1.8in]{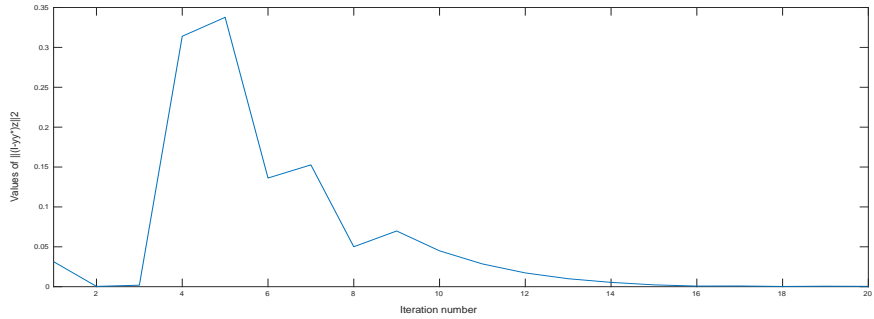} 
\end{center}\vs{-0.3cm}
\caption{Iteration numbers versus $\|(I-yy^\ast)z\|^2$}
\label{fig:4}
\end{figure}
\begin{figure}[!htb]
\begin{center}
\hfill \includegraphics[width = 5.5in,height=1.8in]{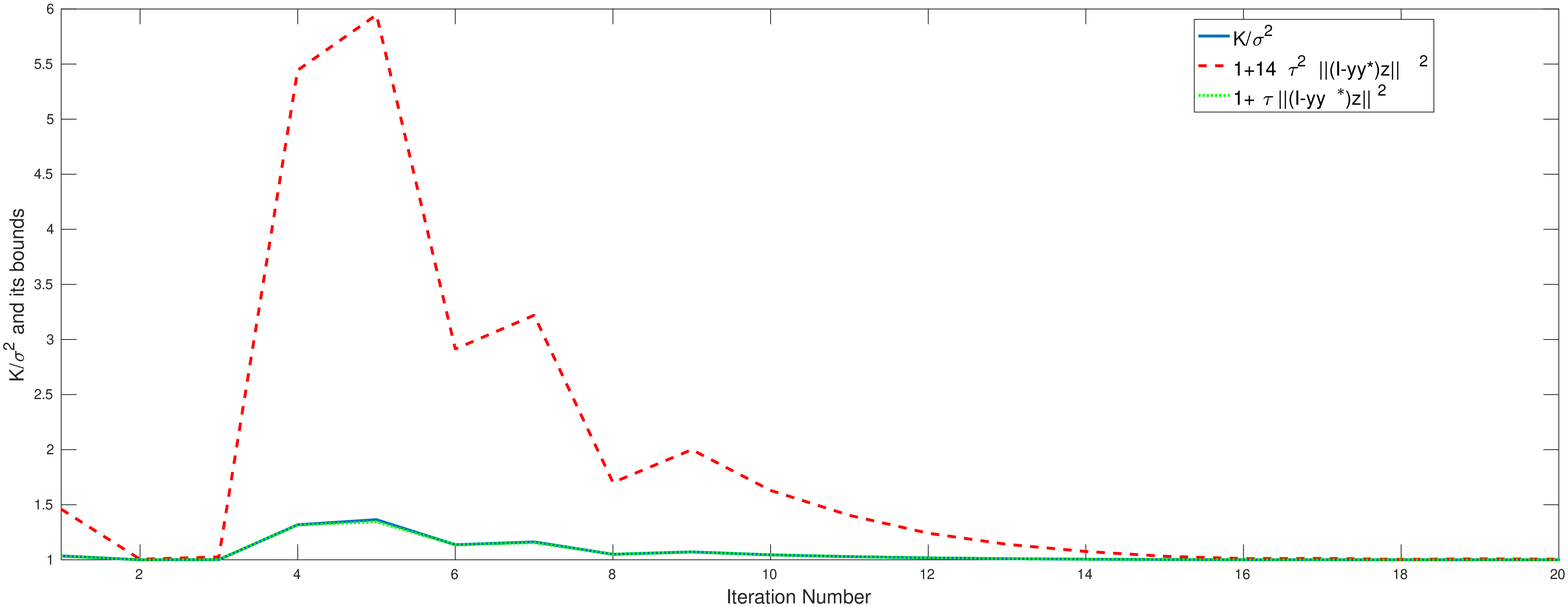} 
\end{center}\vs{-0.4cm}
\caption{Iteration numbers versus $K/\sigma^2$ and its bounds}
\label{fig:5}
\end{figure}
\begin{figure}[!htb]
\begin{center}
\hfill \includegraphics[width = 5.5in,height=1.8in]{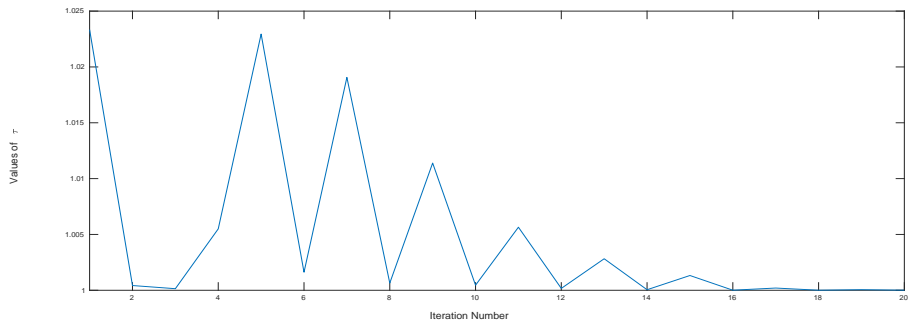} 
\end{center}\vs{-0.4cm}
\caption{Iteration numbers versus $\tau$}
\label{tauf2}
\end{figure}

\vs{0.2cm}The Figures-\ref{fig:4}, \ref{fig:5}, and \ref{tauf2} shows the curves for
 $\|(I-yy^\ast)z\|^2,$ $K/\sigma^2,$ and $\tau$  against restart number respectively. Observe from the Figure-\ref{fig:4} that norm of $(I-yy^\ast)z$ recedes near to zero as the restart number grows. Thus, by using the Lemma-\ref{flemsv} note that upper and lower bounds for $K/\sigma^2$ nearly coincide when the restart number is larger as $\tau$ is finite. The Figure-\ref{fig:5} demonstrate this fact.  
As in the previous example, It has been observed from the Figure-\ref{tauf2} that $\tau < 1.025.$   However, a theoretical result that gives a real number upper bound for $\tau$  has yet to be found.
%gives a real number upper bound for $\tau$is to find yet.

\section{Conclusions}
In this paper, bounds for a ratio of residual norms in the refined and Rayleigh-Ritz projections have been
derived. These bounds are in terms of $\|(I-yy^\ast)z\|^2.$ and$\tau,$ a scalar in the line search and least sqaures method; See equation (\ref{E3}).  Here, $Vy$ is an eigenvector approximation in the Rayleigh-Ritz projection method and $z$ is a solution vector of a least squares problem in the equation (\ref{E2}).

Moreover, the bounds that are derived  in this paper are different from the relationships between the above mentioned residuals which have already been studied by Z. Jia; see Section 4 in Z. Jia ``Some theoretical comparisons of refined Ritz vectors and Ritz vectors'', Science in China Ser. A Mathematics 2004 Vol.47 Supp. 222-233. 
 In this reference, the relationships between the above-mentioned residuals are in terms of the angle between refined Ritz vector and Ritz vector and the second smallest singular value of a singular value problem in the refined method. Thus, computing those bounds practically requires the computation of a Ritz vector, refined Ritz vector, the angle between them and the second smallest singular value. It is very costly to compute all these quantities. Thus, these relations are useful only \underline{theoretically} since once refined Ritz vector and Ritz vectors are computed, practically there is no requirement of  computing the second smallest singular value to compare the   residual norms in both the methods. 

The bounds derived in this article for the ratio of residual norms in the Rayleigh-Ritz and the refined projection methods are \underline{practically useful.} These bounds predicts how much smaller the residual norm in refined projection method compared to residual norm in the Rayleigh-Ritz method, without computing the refined Ritz vector.


\begin{thebibliography}{99}
\bibitem{feng} S. Feng and Z. Jia, A Refined Jacobi-Davidson method and its correction equation, {\em Computers and Mathematics with applications}, 49, 417-427, 2005.
\bibitem{jiac} Z. Jia, The convergence of generalized Lanczos methods for large unsymmetric eigenproblems,  {\em SIAM J. Matrix. Anal. Appl.}, 16:3, 843-862, 1995.
\bibitem{HP1} R.B. Morgan and M. Zeng, Harmonic projection methods for large non-symmetric eigenvalue
problems, \textit{Numer. Linear algebra Appl.}, 5:1, 33-55, 1998.
\bibitem{someshwaragift} M. Ravibabu and A. Singh, On Refined Ritz vectors and polynomial characterization,  {\em Comp. and Math. with Appl.,} 67, 1057-1064, 2014.
\bibitem{Someshwaragift2} M. Ravibabu and A.Singh, A new variant of Arnoldi method for approximation of eigenpairs, \textit{Journal of Computational and Applied Mathematics,} 344, 424-437. https://doi.org/10.1016/j.cam.2018.05.047
\bibitem{Mallannagift} M. Ravibabu and A. Singh, A least squares and line search variant of the Jacobi-Davidson method, \textit{Under communication}.
%\bibitem{KVSLNSMKVB} M.Ravibabu, Computational bounds for the ratio of residual norms in Rayleigh-Ritz and refined projection methods:II,  \textit{Under communication}.
\bibitem{saad} Y. Saad, {\em Numerical Methods for Large Eigenvalue Problems}, Second Edition, SIAM, 2001.
\bibitem{slei} G.L. G. Sleijpen and H. A. Van der Vorst, A Jacobi-Davidson iteration method for linear eigenvalue problems, {\em SIAM REVIEW}, 42:2, 267-293, 2000.
\bibitem{ste}  G.W. Stewart, {\em Matrix Algorithms: Vol II, eigensystems}, SIAM, Philadelphia, PA, 2001.
\bibitem{HPrec} G. Wu, The Convergence of Harmonic Ritz Vectors and Harmonic Ritz Values,
Revisited, \textit{SIAM. J. Matrix Anal. Appl.}, 38(1), 118-133,
2017.
\end{thebibliography}
\end{document}